\documentclass[twoside,notitlepage,11pt]{amsart}

\usepackage{amssymb}
\usepackage[leqno]{amsmath}
\usepackage{amsfonts}
\usepackage{amsopn}
\usepackage{amstext}
\usepackage{amsthm}

\usepackage[colorlinks,backref]{hyperref}

\newcommand{\Cee}{{\cal{C}}}

\newcommand{\Ef}{{\cal{F}}}

\newcommand{\Pee}{{\cal{P}}}
\newcommand{\See}{{\cal{S}}}

\newcommand{\Nat}{{\Bbb{N}}}
\newcommand{\Qyu}{{\Bbb{Q}}}
\newcommand{\Qyuc}{{\cal{Q}}}
\newcommand{\Err}{{\Bbb{R}}}

\newcommand{\al}{\alpha}

\renewcommand{\phi}{\varphi}
\renewcommand{\rho}{\varrho}

\newcommand{\rest}{\restriction}

\newcommand{\ntr}{n\in\omega}

\newcommand{\loe}{\leq}
\newcommand{\goe}{\geq}

\newcommand{\subs}{\subseteq}
\newcommand{\sups}{\supseteq}
\newcommand{\nnempty}{\ne\emptyset}

\renewcommand{\iff}{\Longleftrightarrow}

\newcommand{\cl}{\operatorname{cl}}

\newcommand{\id}[1]{\operatorname{id}_{#1}}

\newcommand{\dom}{\operatorname{dom}}

\newcommand{\io}{\operatorname{io}}

\newcommand{\by}{/}

\newcommand{\Sep}{\operatorname{Sep}}

\newtheorem{tw}{Theorem}[section]

\newtheorem{lm}[tw]{Lemma}
\newtheorem{prop}[tw]{Proposition}

\newtheorem{conj}[tw]{Conjecture}
\newtheorem{question}[tw]{Question}
\theoremstyle{definition}

\theoremstyle{remark}

\newcommand{\setof}[2]{\{#1\colon #2\}}

\newcommand{\sett}[2]{\{#1\}_{#2}}
\newcommand{\sn}[1]{\{#1\}} % singleton
\newcommand{\dn}[2]{\{#1,#2\}} % doubleton
\newcommand{\map}[3]{#1\colon #2 \to #3} % A function
\newcommand{\img}[2]{#1[#2]} % image of a set
\newcommand{\inv}[2]{{#1}^{-1}[#2]} % preimage of a set

\providecommand{\cal}{\mathcal}
\renewcommand{\Bbb}{\mathbb}

\newcommand{\suppt}{\operatorname{suppt}}

\providecommand{\nat}{\omega}

\newcommand{\im}{\operatorname{im}}
\newcommand{\bal}{\operatorname{B}}
\newcommand{\clbal}{\overline{\bal}}

\newcommand{\norm}[1]{\|#1\|}

\newcommand{\uball}[1]{\clbal_{#1}}
\newcommand{\ubal}{\uball}

\newcommand{\re}{\operatorname{Re}}

\newcommand{\weakstar}{\ensuremath{{weak}^*}}

\newcommand{\Sp}{\operatorname{span}}

\newcommand{\cmp}{\circ} % composition!!!

\title{Complementation in spaces of continuous functions on compact lines}
\author{
Ond\v rej F.K. Kalenda and
Wies{\l}aw Kubi\'s
}
\address{(O.F.K. Kalenda) Department of Mathematical Analysis \\Faculty of Mathematics and Physic\\ Charles University\\Sokolovsk\'{a} 83, 186~75\\Praha 8, Czech Republic}
\email{kalenda@karlin.mff.cuni.cz} 
\thanks{The first author was supported in part by the grant GAAV IAA 100 190 901  and in part by the Research Project MSM~0021620839 from the Czech Ministry of Education.}
\address{(W. Kubi\'s) Institute of Mathematics of the Academy of Sciences of the Czech Republic, \v Zitn\'a 25, 115 67 Praha 1, Czech Republic
\and
Institute of Mathematics, Jan Kochanowski University, \'Swietokrzyska 15, 25-406 Kielce, Poland}
\thanks{The second author was supported in part by the Grant IAA 100 190 901 and by the Institutional Research Plan of the Academy of Sciences of Czech Republic No. AVOZ 101 905 03.}
\email{kubis@math.cas.cz}

\begin{document}

\begin{abstract}
We characterize order preserving continuous surjections between compact linearly ordered spaces which admit an averaging operator, together with estimates of the norm of such an operator. This result is used to the study of strengthenings of the separable complementation property in spaces of continuous functions on compact lines. These properties include in particular continuous separable complementation property and existence of a projectional skeleton.
\end{abstract}


\maketitle

% Temporary:
%\tableofcontents

\section{Introduction: separable complementation properties and semilattice of separable subspaces}

A {\em compact line} is, by definition, a linearly ordered compact space.
We study separable complementation property and its strengthenings in Banach spaces of continuous functions on compact lines.
Below we discuss necessary notions and basic facts related to complementation in non-separable Banach spaces.

By $\ubal E$ we denote the closed unit ball of a Banach space $E$. 
If $E=F^*$ then we always consider $\ubal E$ with the \weakstar~topology. 
A {\em projection} is, by definition, a bounded linear operator $\map PEE$ such that $P P = P$. The space $\im P=\setof{Px}{x\in E}$ is then called {\em complemented} in $E$. Note that $x\in \im P$ if and only if $Px=x$.
Let $E,F$ be subspaces of spaces of type $C(K)$, both containing the constant functions. Recall that $\map TEF$ is {\em regular} if $T$ is linear, $T1=1$ and $Tf\goe0$ whenever $f\goe0$.
Note that regular operators necessarily have norm one.

Recall that a Banach space $E$ has the {\em separable complementation property} ({\em SCP} for short) if for every countable set $S\subs E$ there exists a projection $P$ on $E$ such that $S\subs\im P$ and $\im P$ is separable.
A Banach space $E$ has the {\em controlled separable complementation property} if for every countable sets $S,T$ such that $S\subs E$ and $T\subs E^*$ there exists a projection $P$ on $E$ such that $\im P$ is separable, $S\subs \im P$ and $T\subs\im P^*$. If projections appearing in the above definitions have norm $\loe k$ then we shall say that $E$ has the $k$-SCP or controlled $k$-SCP respectively. Similarly, if the projections are regular, we shall say that $E$ has the {\em regular} (controlled) SCP.
For a survey of separable complementation and related properties we refer to \cite{PliYos}.

We will consider also some other strengthenings of the SCP. To see that they are natural we first reformulate the definition of SCP using the semilattice of separable subspaces. For a Banach space $E$ consider the family $\Sep(E)$ of all closed separable linear subspaces. If we consider this family ordered by inclusion, it becomes a $\sigma$-complete semilattice. This means that  any sequence $(S_n)$ of elements of $\Sep(E)$ has a supremum $S\in\Sep(E)$.
Indeed, we can (and have to) take $S=\overline{\Sp \bigcup_n S_n}$. 

Remark now that a Banach space $E$ has the SCP if and only if complemented separable subspaces form a cofinal subset of $\Sep(E)$. We say that $E$ has the {\em continuous SCP} if there is a closed cofinal subset  $\Ef\subs\Sep(E)$ formed by complemented subspaces. The words {\em closed cofinal} mean that:
\begin{itemize}
	\item For each $S\in\Sep(E)$ there is $F\in\Ef$ with $F\sups S$.
	\item Whenever $F_1\subs F_2\subs\dots$ are from $\Ef$, their supremum in $\Sep(E)$ belongs to $\Ef$.
\end{itemize}

If all elements of $\Ef$ are $k$-complemented for some $k\ge1$, we say that $E$ has the continuous $k$-SCP.

A further strengtening of the continuous SCP is the notion of a {\em projectional skeleton}, which is a closed cofinal subset $\Ef\subs\Sep(E)$ together with compatible projections, i.e., for each $S\in \Ef$ we have given a bounded linear projection $P_S:X\to S$ such that $P_S\cmp P_T=P_S$ whenever $S\subs T$.
If all the projections $P_S$ have norm bounded by $k$, we talk about a $k$-projectional skeleton.

The notion of a projectional skeleton was introduced and studied in \cite{skeletons}. The original definition is slightly different, but it follows from \cite[proof of Theorem 4.7]{skeletons} that our definition is equivalent to the original one.

The advantage of the continuous SCP in comparison with the standard SCP is that the closed cofinal subset of $\Sep(E)$ is in a sense uniquely determined. More precisely, we have the following easy lemma:

\begin{lm}\label{spectral}
 Let $E$ be a (nonseparable) Banach space and for each $n\in\Nat$ let $\Ef_n$ be a closed cofinal subset of $\Sep(E)$. Then the intersection $\bigcap_{n\in\Nat}\Ef_n$ is again a closed cofinal subset of $\Sep(E)$.
\end{lm}

In particular, if there is a closed cofinal subset of $\Sep(E)$ formed by non-complemented subspaces, then $E$ does not have the continuous SCP.
Further, if we want to decide whether a given space has the continuous SCP it is enough to test complementability of spaces from a given closed cofinal subset of $\Sep(E)$. This is important for example in the case of $C(K)$ spaces.

It will be helpful to introduce the following terminology. Let $E$ be a (nonseparable) Banach space and $\Pee$ a property. We will say that {\em typical separable subspace of $E$ has property $\Pee$} if there is a closed cofinal subset $\Ef\subs\Sep(E)$ such that each element of $\Ef$ has property $\Pee$.

Let $E=C(K)$ where $K$ is a compact space. Suppose that $q:K\to L$ is a continuous surjection of $K$ onto a compact space $L$. Denote by $q^*$ the mapping from $C(L)$ into $C(K)$ defined by $q^*f=f\cmp q$. Then $q^*$ is an isometry and $q^*C(L)$ is a closed subspace of $C(K)$. Moreover, it is separable if and only if $L$ is metrizable. 

\begin{lm}\label{typical_is_quotient} Let $K$ be a compact space. Then typical separable subspace  of $C(K)$ is of the form $q^*C(L)$, where $L$ is a metrizable compact space and $q:K\to L$ a continuous surjection.

If $K$ is, moreover, a compact line, then typical separable subspace  of $C(K)$ is of the form $q^*C(L)$, where $L$ is a metrizable compact line and $q:K\to L$ a continuous order preserving surjection. 
\end{lm}

\begin{proof} It follows from the Stone-Weierstrass theorem that $S\in\Sep(C(K))$ is of the form $q^*C(L)$ if and only if $S$ is a closed algebra containing constant functions (and closed to taking complex conjugate in the complex case).
It is clear that these spaces form a closed cofinal subset.

The second part will be proved in the beginning of Section~\ref{typq}.
\end{proof}

It is well-known and easy to show that whenever a Banach space $E$ has the SCP, it has the $k$-SCP for some $k\ge 1$. An analogous statement holds also for projectional skeletons, see \cite[Proposition 4.1]{skeletons}.

We shall present (Theorem~\ref{tunboundedscp} below) an example of a Banach space with ``unbounded" continuous SCP. Namely, it is a $C(K)$ space, where $K$ is a certain compact line; the space $C(K)$ has the continuous SCP, yet it has no continuous $k$-SCP for any $k\in\nat$.
 
It is clear that the existence of a $k$-projectional skeleton implies continuous $k$-SCP which implies $k$-SCP. It is worth to notice that there is a partial converse proved in \cite[Lemma 6.1]{K2006}:

\begin{prop} Let $E$ be a Banach space of density $\aleph_1$ which has continuous $1$-SCP. Then $E$ has $1$-projectional skeleton.
\end{prop}

It is natural to ask the following question:

\begin{question} Let $E$ be a Banach space with the continuous $k$-SCP. Does it have a projectional skeleton?
\end{question}

We finish by a short discussion of the controlled SCP. The following simple fact has already been noticed by several authors (see e.g. \cite[Cor. 1]{Ferrer} for the case of $C(K)$ spaces).

\begin{prop}\label{owenpqwjr}
Assume $E$ has the controlled SCP. Then the dual unit ball of $E$ is $\aleph_0$-monolithic.
\end{prop}

\begin{proof}
Fix a countable $S\subs \bal(E^*)$ and, using controlled SCP, find a projection $\map PEE$ such that $\im(P)$ is separable and $S\subs\im(P^*)=:F$. Then $F$ is a \weakstar~closed linear subspace of $E^*$ such that bounded subsets are \weakstar~metrizable. Hence $\cl_*(S)\subs F\cap\bal(E^*)$ is \weakstar~metrizable.
\end{proof}

Let us remark that in the above proposition the controlled SCP cannot be replaced by SCP. Indeed, let $E=\ell_1([0,1])$. Then $E$ has $1$-SCP (even a $1$-projectional skeleton), but the dual unit ball is homemorphic to $[-1,1]^{[0,1]}$ which is not $\aleph_0$-monolithic as it is a separable nonmetrizable compact space. 	
Proposition~\ref{owenpqwjr} implies that if a $C(K)$ space has the controlled SCP, then $K$ itself must be $\aleph_0$-monolithic.
In the case of compact lines, we have a much better result:

\begin{tw}\label{twergoq}
Let $K$ be a compact line. The following properties are equivalent.
\begin{enumerate} 
	\item[(a)] $C(K)$ has the regular controlled SCP.
	\item[(b)] $C(K)$ has the controlled SCP.
	\item[(c)] $C(K)$ has the SCP.
	\item[(d)] $K$ is $\aleph_0$-monolithic.
\end{enumerate}
\end{tw}

The above result will be proved in Section~\ref{scteaaofcl}.
Note that the implication (b)$\implies$(d) is valid for arbitrary compacta (see \cite[Cor. 1]{Ferrer}).

We remark that the following general question on a possible full converse of Propositon~\ref{owenpqwjr} seems to be open.

\begin{question} Let $E$ be a Banach space such that $\overline B_{E^*}$ is $\aleph_0$-monolithic. Does $E$ have the SCP? How about controlled SCP?
\end{question}

\section{Extension and averaging operators for compact lines}\label{scteaaofcl}

In this section we study extension and averaging operators for compact lines. The results will be first used for proving Theorem~\ref{twergoq} and then applied in the next section to the study of (controlled) separable complementation properties for $C(K)$ spaces.
We shall need the following classical fact on measures on compact lines. For completeness, we give an elementary proof.

\begin{lm}\label{lmsuslinmeasureAa}
Let $K$ be a compact line and assume that there exists a strictly positive Borel measure $\mu$ on $K$, i.e. $\mu(J)>0$ whenever $J$ is an interval with nonempty interior in $K$. Then $K$ is separable.
\end{lm}

\begin{proof}
Let $A$ be the set of all atoms of $\mu$, that is $A = \setof{p\in K}{\mu(\sn p)>0}$.
Then $A$ is countable.
For each $n>0$ let $\Ef_n$ be a maximal family of pairwise disjoint intervals $J\subs K$ satisfying $J\cap A = \emptyset$ and $\mu(J)<1/n$.
Clearly, $\Ef_n$ is finite, since $\mu(K) < +\infty$.
Fix a countable set $D\sups A$ so that $D\cap J\nnempty$ for every $J\in\bigcup_{n>0}\Ef_n$. We claim that $D$ is dense.

Suppose otherwise and choose a nonempty open interval $U\subs K$ with $U\cap D=\emptyset$.
Then $U\cap A=\emptyset$, therefore we can find nonempty pairwise disjoint open intervals $W_0, W_1, W_2 \subs U$.
Assume $W_1$ is above $W_0$ and below $W_2$, i.e. $x_0 < x_1 < x_2$ whenever $x_i\in W_i$ for $i<3$.
Find $\ntr$ so that $\mu(W_i) \goe 1/n$ for $i<3$.
Observe that $W_1 \cap \bigcup\Ef_n\nnempty$ since otherwise we would be able to find a subinterval $G$ of $W_1$ satisfying $\mu(G)<1/n$, which could be added to the family $\Ef_n$, contradicting its maximality.
Fix $J\in \Ef_n$ such that $J\cap W_1\nnempty$.
Fix $x\in D\cap J$.
Since $D\cap U=\emptyset$, we conclude that either $x$ is below $W_0$ or above $W_2$. Thus., $W_i\subs J$ for some $i\in \dn 02$.
Hence, $\mu(J)\goe \mu(W_i) \goe 1/n$, a contradiction.
\end{proof}

The above lemma can also be proved as follows: assuming $K$ is not separable, it must be a Suslin line. On the other hand, the product measure shows that $K\times K$ satisfies the Suslin condition, i.e. every disjoint family of open sets is countable. This is a contradiction, since the square of a Suslin line fails the Suslin condition (see e.g. \cite{Kunen} for details).
On the other hand, it is worth mentioning that the above lemma holds, assuming only that $\mu$ is finitely additive.

If $K$ is a compact space and $L$ is a closed subset of $K$, then an {\em extension operator} is a bounded linear operator $T:C(L)\to C(K)$ satisfying $Tf\rest_L=f$ for each $f\in C(L)$. Such an operator need not exist in general. However, for compact lines we have the following lemma. It is proved in \cite[Lemma 4.2]{K2006}.

\begin{lm}\label{line-extension} Let $K$ be a compact line and $L\subs K$ be a closed subspace. Then there is a regular extension operator $T:C(L)\to C(K)$.
\end{lm}

We will use this in the proof of Theorem~\ref{twergoq}.

\begin{lm}\label{wraasjJFWRQ}
Let $K$ be a nonempty $\aleph_0$-monolithic compact line. Then $C(K)$ has the regular controlled SCP. \end{lm}

\begin{proof} Fix countable sets $D\subs C(K)$ and $G\subs C(K)^*$.
Fix $f\in C(K)$. We say that $p\in K$ is {\em irrelevant} for $f$ if
either $p=\min K$ and $f$ constant on $[p,b]$ for some $b>p$, or else
$p=\max K$ and $f$ is constant on $[a,p]$ for some $a<p$, or else
$\min K<p<\max K$ and $f$ is constant on $[a,b]$ for some $a<p<b$.
Denote by $X_f$ the set of all $p\in K$ which are not irrelevant for $f$. It has been observed in \cite[proof of Proposition 3.4]{K2006} that $X_f$ is separable and closed in $K$.

Now fix $\mu\in C(K)^*$ and let $Y=\suppt(\mu)$ be the support of $\mu$, i.e. the set of all points $p\in K$ such that $|\mu|(U)>0$ for every neighborhood $U$ of $p$. Then $Y$ is a closed subset of $K$ which satisfies the countable chain condition. Moreover, $|\mu|$ is strictly positive on $Y$. By Lemma~\ref{lmsuslinmeasureAa}, it follows that $Y$ is separable.
Define
$$X=\cl_K\Bigl(\bigcup_{f\in D}X_f \cup \bigcup_{\mu\in G}\suppt(\mu)\Bigr).$$
By the above arguments, $X$ is separable. Hence $X$ is second countable, because $K$ is $\aleph_0$-monolithic. Lemma~\ref{line-extension} says that there exists a regular extension operator $\map T{C(X)}{C(K)}$ such that for every $f\in C(X)$, all relevant points of $Tf$ are in $X$. In other words, $T$ is a linear operator satisfying $(Tf)\rest X=f$ for $f\in C(X)$, $T1=1$ and $Tf\goe0$ whenever $f\goe 0$. In particular $\norm T\loe 1$. Moreover, $Tf$ is constant on every interval disjoint from $X$. It follows that $T(f\rest X)=f$ whenever $X_f\subs X$. Let $P = T R$, where $\map R{C(K)}{C(X)}$ denotes the restriction operator $f\mapsto f\rest X$. Then $Pf=f$ for every $f\in D$ and $\im(P)$ is linearly isometric with $C(X)$, therefore it is separable. Given $\mu\in G$, we have that $\suppt(\mu)\subs X$, therefore $P^*\mu=\mu$. This shows that $G\subs \im(P^*)$ and completes the proof.
\end{proof}

The next lemma is perhaps of independent interest. It gives a criterion for the failure of SCP.
Recall that a Banach space is \emph{weakly compactly generated} (briefly: WCG) if it contains a weakly compact linearly dense set.

\begin{lm}\label{lnonSCPviafakeWCG}
Let $E$ be a Banach space such that $E\by F$ is WCG for some separable subspace $F\subs E$, while $E$ is not embeddable into any WCG space. Then no Banach space containing $E$ has the separable complementation property.
\end{lm}

\begin{proof}
Suppose $X$ is a Banach space containing $E$ and $Y\subs X$ is separable, complemented in $X$ and such that $F\subs Y$.
Let $Z$ be the closure of $E+Y$.
Note that
$$Z\by E = (Z\by F) \by (E\by F)$$
and $E\by F$ is WCG, by the assumption.
On the other hand, $Z \by E$ is separable, because $Y$ is so.
We conclude that $Z\by F$ is WCG, because it is generated by $K\cup S$, where $K$ is a weakly compact set generating $E\by F$
and $S=\sn0\cup\setof{s_n}{\ntr}$ is such that
$\norm{s_n}\loe1/n$ for $\ntr$ and $\sett{s_n + E}{\ntr}$ is linearly dense in $Z\by E$.
Further, $Z\by Y$ is WCG, being the quotient of a WCG space $Z\by F$.
Finally, $Z = Y \oplus (Z\by Y)$ is WCG.
It follows that $E$ is embeddable into a WCG space, a contradiction.
\end{proof}

The above lemma will be applied to a particular space $E = C(K)$, where $K$ is a double arrow line.
Recall that $K$ is a \emph{double arrow line} if it has a two-to-one continuous increasing map onto the unit interval $[0,1] \subs \Err$.
More specifically, a double arrow line is, up to order isomorphism, a line of the form
$$D(A) = ([0,1]\times\sn0) \cup (A\times \sn1) \subs [0,1] \times \dn01,$$
endowed with the lexicographic ordering.
Note that $D(A)$ is metrizable if and only if the set $A\subs [0,1]$ is countable. On the other hand, $D(A)$ is always separable.
Let $K = D(A)$ and let $\map q{D(A)}{[0,1]}$ be the canonical increasing quotient map. Let
$$F := q^* C([0,1]) = \setof{f\cmp q}{f\in C([0,1])}.$$
It is well known and easy to check that $C(K) \by F$ is isomorphic
to $c_0(A)$.
On the other hand, $c_0(A)$ is a canonical example of a WCG Banach space, generated by a weakly compact set homeomorphic to the one-point compactification of a discrete space of size $|A|$.
Finally, if $|A| > \aleph_0$ then $C(K)$ is not embeddable into any WCG space, because subspaces of WCG spaces are weakly Lindel\"of, while $C(K)$ is not (all non-constant increasing functions into $\dn 01$ form an uncountable discrete subset of $C(K)$ in the weak topology).
The argument above goes back to Corson~\cite{Corson61}.
%cogolo

The following fact belongs to the folklore. For completeness, we give a proof.

\begin{lm}\label{ldblarrownonmon}
Let $K$ be a separable and non-metrizable compact line. Then $K$ contains a copy of a non-metrizable double arrow line.
\end{lm}

\begin{proof}
Define the following equivalence relation on $K$:
$$x \sim y \iff [x,y] \text{ is non-metrizable}.$$
It is easy to check that $\sim$ is indeed an equivalence relation and that its classes are convex.
As $K$ is separable, it is also first countable and hence the $\sim$-equivalence classes are closed.
Let $J = K\by \sim$ and let $\map qKJ$ be the canonical quotient map.
Observe that $J$ is not a singleton, since $K$ is not metrizable.
Further, $J$ is connected.
Indeed, if $a,b\in J$ were such that $a < b$ and $[a,b] = \dn ab$, then taking $s = \max q^{-1}(a)$, $t = \min q^{-1}(b)$, we would have $s \not\sim t$ and $[s,t] = \dn st$, which contradicts the definition of $\sim$.
It follows that $J$ is order isomorphic to the standard unit interval $[0,1] \subs \Err$.

Now let $L\subs K$ be a minimal closed set such that $f = q \rest L$ is onto.
Note that $f$ is two-to-one, since given $t\in J$, removing the interior of the interval $f^{-1}(t)$ we obtain a smaller closed set that maps onto $J$.
We conclude that $L$ is a double arrow, order isomorphic to $D(A)$, where $A=\setof{t\in J}{|f^{-1}(t)|>1}$.
\end{proof}

\begin{proof}[Proof of Theorem~\ref{twergoq}]
Implications (a)$\implies$(b)$\implies$(c) are trivial and (d)$\implies$(a) is the content of Lemma~\ref{wraasjJFWRQ}. It remains to show that (c)$\implies$(d).

Fix a compact line $K$ which is not $\aleph_0$-monolithic. We shall prove that $C(K)$ fails the SCP.
Fix a separable and non-metrizable closed subset $K_0$ of $K$.
By Lemma~\ref{ldblarrownonmon}, $K_0$ contains a non-metrizable double arrow $L$.
By Lemma~\ref{line-extension}, $C(L)$ is embeddable into $C(K)$.
Further, $C(L)$ is not embeddable into any WCG space, yet it contains a natural copy $F$ of $C([0,1])$ such that $C(L)\by F$ is WCG (being isomorphic to $c_0(A)$ for some set $A$).
Finally, by Lemma~\ref{lnonSCPviafakeWCG}, $C(K)$ fails to have the SCP.
\end{proof}

For averaging operators the situation is more complicated. Recall that {\it an averaging operator} of a quotient mapping $q:K\to L$ of a compact space $K$ onto a compact space $L$ is a bounded linear operator $T:C(K)\to C(L)$ satisfying
$T(f)\circ q=f$ for each $f\in C(K)$. The existence of an averaging operator is thus equivalent to the complementability of the canonical copy of $C(L)$ (i.e., of $q^*C(L)$, see the previous section) in $C(K)$. We will characterize the existence of an averaging operator for an order preserving quotient of compact lines.

To this end we need two notions which will be used throughout this paper.

Firstly, let $K$ be a compact line and $x\in K$. We say that $x$ is an {\it external} point of $K$, if $x$ is an isolated point either of $[x,\to)$ or of $(\leftarrow,x]$. If $x$ is not external, i.e. if it is isolated neither in $[x,\to)$ nor in $(\leftarrow,x]$, we say that $x$ is an {\it internal} point of $K$.

Secondly, let $X$ be a linearly ordered set and $A\subs X$. We shall define an order for points of $A$ and an order for $A$. We will call that order {\em internal order}. The definition reads as follows.  If $A=\emptyset$, let $\io(A)=-1$. 
Further, if $A$ is nonempty, we define $\io(x,A)$ for each $x\in A$ inductively:
\begin{itemize}
	\item $\io(x,A)\ge0$  for each $x\in A$;
	\item $\io(x,A)\ge n$ if both sets
	\begin{align*} A_{n-1}^-&=\{y\in A: y<x \mbox{ and } \io(y,A)\ge n-1 \}\\
	A_{n-1}^+&=\{y\in A: y>x \mbox{ and } \io(y,A)\ge n-1\} \end{align*}
	are nonempty and $\sup A_{n-1}^-=\inf A_{n-1}^+=x$.
\end{itemize}
%\begin{align*}
%\io(x,A)\ge0 & \mbox{ for each }x\in A,\\
%\io(x,A)\ge n & \mbox{ if } \exists y<x \,\exists z>x\, \forall u\in[y,x)\,\forall v\in (x,z] \\ & \qquad \exists r\in(u,x)\,\exists s\in (x,v):\io(s,A)\ge n-1 \& \io(r,A)\ge n-1.
%\end{align*}
Now,
$$\io(x,A)=\sup\{n\in\Nat\cup\{0\}: \io(x,A)\ge n\},\qquad
\io(A)=\sup\{\io(x,A):x\in A\}.$$ 
Therefore, $\io(x,A)$ can take values in $\Nat\cup\{0,\infty\}$ and $\io(A)$ takes value in $\Nat\cup\{-1,0,\infty\}$. It would be possible to allow also infinite ordinal values for the internal order, but it is not needed for our purposes.

We will need the following result on additivity of the internal order:

\begin{lm}\label{io-additivity} Let $X$ be a linearly ordered set and $A,B\subs X$. Then
$$\io(A\cup B)\le \io(A)+\io(B)+1.$$
\end{lm}

\begin{proof} Let $n\ge-1$ be such that $\io(A\cup B)\ge n$. We will prove by induction on $n$ that $\io(A)+\io(B)+1\ge n$ as well.

If $n=-1$, then obviously $\io(A)+\io(B)+1\ge -1+(-1)+1=-1$. If $n=0$, then $A\cup B$ is nonempty, therefore at least one of the sets $A$ and $B$ is nonempty, too. Hence the inequality follows.

Now suppose that $n>0$ and that the assertion is valid for all smaller values.
Suppose that $\io(A\cup B)\ge n$. Then there is $x\in A\cup B$ such that $\io(x,A\cup B)\ge n$. The point $x$ belongs to one of these sets, without loss of generality suppose $x\in A$. Take any $y>x$ and set
$$A^y=A\cap (x,y),\qquad B^y=B\cap (x,y).$$
It follows from the assumption $\io(x,A\cup B)\ge n$ that such a $y$ does exist and that for each choice of such a $y$ we have $\io(A^y\cup B^y)\ge n-1$. By the induction hypothesis we get
$$\io(A^y)+\io(B^y)+1\ge n-1.$$
If $y'\in(x,y)$, then clearly $\io(A^{y'})\le\io(A^y)$ and $\io(B^{y'})\le\io(B^y)$. As the values of the internal order are only from $\Nat\cup\{-1,0,\infty\}$, there is $y_0>x$ such that for $y\in (x,y_0)$ the values of $\io(A^y)$ and $\io(B^y)$ do not depend on $y$. 

Similarly we can proceed for $z<x$ -- define $A_z$ and $B_z$ and find $z_0<x$  
such that for $z\in(z_0,x)$ the values of $\io(A_z)$ and $\io(B_z)$ do not depend on $z$. Fix any $z\in(z_0,x)$ and $y\in(x,y_0)$. Suppose without loss of generality that $\io(A_z)\le\io(A^y)$. Then
$$\io(x,A)\ge\io(A_z)+1,$$
and so
$$\io(A)+\io(B)+1\ge \io(x,A)+\io(B)+1\ge \io(A_z)+1+\io(B_z)+1\ge n,$$
which was to be proved.
\end{proof}

Below is the crucial lemma involving internal order. It improves \cite[Lemma~4.1]{K2006}.
A similar negative result, for certain maps of $0$-dimensional metric compacta, can be found in the last chapter of Pe\l czy\'nski's dissertation~\cite{Pel-diss}.

\begin{lm}\label{lines-averaging} Let $K$ and $L$ be compact lines and $q:K\to L$ be an order preserving continuous surjection. Set
$$Q=\{x\in L:x\mbox{ is an internal point of }L\mbox{ and }|q^{-1}(x)|>1\}.$$
%
%and denote by $q^*C(L)$ the canonical copy of $C(L)$ in $C(K)$, i.e.,
%
%$$q^*C(L)=\{f\cmp q:f\in C(L) \}.$$
Then we have the following:
\begin{enumerate}
	\item $q^*C(L)$ is $1$-complemented in $C(K)$ if and only if $Q=\emptyset$. In this case $q$ admits a right inverse and hence also a regular averaging operator.
	\item $q^*C(L)$ is complemented in $C(K)$ if and only if $\io(Q)<\infty$.
	More specifically:
	
\begin{itemize}
	\item[(a)] If\/ $\io(Q)\ge n$, then any projection $P$ of $C(K)$ onto $q^*C(L)$ satisfies $\|P\|\ge2+\lceil\frac{n-1}2\rceil$.
	\item[(b)] If\/ $\io(Q)\le n$, then there is a projection $P$ of $C(K)$ onto $q^*C(L)$ with $\|P\|\le 2n+3$.
\end{itemize}
\end{enumerate}
\end{lm}

\begin{proof}
We start by proving assertion 2(a).
Suppose that $\io(Q)\ge n$ and that $P$ is a bounded linear projection  of $C(K)$ onto $q^*C(L)$. We will modify and refine the proof of \cite[Lemma 4.1]{K2006}.

For each $x\in L$ the inverse image $q^{-1}(x)$ is a closed interval in $K$ which will be denoted by $[x^-,x^+]$. Then $q^*C(L)$ are exactly those functions from $C(K)$ which are constant on each interval of the form $[x^-,x^+]$ with $x\in L$.

For each $p\in Q$ choose a continuos function $\chi_p:K\to[-1,1]$ such that
$$\chi_p(x)=\begin{cases} -1 & x\le p^-\\ 1 & x\ge p^+ \end{cases}$$
(such a fucntion exists due to the Urysohn lemma) and set $h_p=P(\chi_p)$. Denote by $\bar h_p$ the unique function from $C(L)$ such that $h_p=\bar h_p\cmp q$. Fix any $\delta\in(0,1)$ and set
$$Q^-=\{ p\in Q: \re \bar h_p(p)<\delta\}, \quad Q^+=\{p\in Q:  \re \bar h_p(p)>-\delta\}.$$
Then $Q=Q^-\cup Q^+$ and hence by Lemma~\ref{io-additivity} we have
$\io(Q^-)+\io(Q^+)+1\ge n$. Therefore at least one of the numbers $\io(Q^-)$, $\io(Q^+)$ is at least $m=\lceil\frac{n-1}2\rceil$ (we are using that these numbers are integers). 

Suppose first that $\io(Q^-)\ge m$.
For $p\in Q^-$ set
$$U_p^-=\{x\in L: \re \bar h_p(x)<\delta\}.$$
Then $U_p^-$ is an open set containing $p$. We will choose points
$$p_0<p_1<\dots<p_m$$
in $Q^-$ such that for each $j=0,\dots,m$ we have $\io(p_j,Q^-)\ge m-j$ 
and $p_j\in U_{p_0}^-\cap\dots U_{p_{j-1}}^-$. This is possible by the definition of the internal order using the fact that $\io(Q^-)\ge m$.
Further choose $p_{m+1}>p_m$ such that $p_{m+1}\in\bigcap_{j=0}^m U_{p_j}^-$. This is possible as $p_m$ is an internal point.

Finally choose  a continuous function $f:K\to[-1,1]$ such that $f\rest_{[p_j^-,p_j^+]}=\chi_{p_j}\rest_{[p_j^-,p_j^+]}$ for $j=0,\dots,m$, $f(p_{m+1}^-)=-1$ and $f$ is constant on $[p^-,p^+]$ for $p\in L\setminus\{p_0,\dots,p_m\}$. Such a function can be constructed as follows:
For each $j\in\{0,\dots,m-1\}$ choose by Urysohn lemma a continuous function $u_j:L\to[-1,1]$ such that $u_j(p_j)=1$ and $u_j(p_{j+1})=-1$ and set
$$f(x)=\begin{cases} -1, & x\le p_0^-, \\ \chi_{p_j}(x), & x\in [p_j^-,p_j^+], j=0,\dots,m, \\ u_j(q(x)), & x\in [p_{j}^+,p_{j+1}^-], j=0,\dots,m, \\
-1, & x\ge p_{m+1}^-.\end{cases}$$
If we set $g=f-\sum_{j=0}^m \chi_{p_j}$, we get $g\in q^*C(L)$, i.e.  $Pg=g$. Thus
$$Pf=g+\sum_{j=0}^m h_{p_j}=f-\sum_{j=0}^m\chi_{p_j}+\sum_{j=0}^m h_{p_j}.$$
As $\|f\|=1$ we have
\begin{align*}
-\|P\|&\le - \|P(f)\|\le -|P(f)(p_{m+1}^-)|\le \re P(f)(p_{m+1}^-) \\&
=f(p_{m+1}^-)-\sum_{j=0}^m \chi_{p_j}(p_{m+1}^-)+\sum_{j=0}^m \re \bar h_{p_j}(p_{m+1}) \\ &
\le -1-(m+1)+(m+1)\delta= -2-m+(m+1)\delta.\end{align*}
It follows that $\|P\|\ge2+m-(m+1)\delta$.

If $\io(Q^+)\ge m$, we proceed similarly: We define $U_p^{+}=\{x\in L:\re \bar h_p(x)>-\delta\}$ and choose points $p_0>p_1>\dots>p_m$ in $Q^+$ such that
for each $j=0,\dots,m$ we have $\io(p_j,Q^+)\ge m-j$ 
and $p_j\in U_{p_0}^+\cap\dots U_{p_{j-1}}^+$. Further we choose $p_{m+1}<p_m$ such that $p_{m+1}\in\bigcap_{j=0}^m U_{p_j}^+$. Finally  we choose  a continuous function $f:K\to[-1,1]$ such that $f\rest_{[p_j^-,p_j^+]}=\chi_{p_j}\rest_{[p_j^-,p_j^+]}$ for $j=0,\dots,m$, $f(p_{m+1}^+)=1$ and $f$ is constant on $[p^-,p^+]$ for $p\in L\setminus\{p_0,\dots,p_m\}$. By analogous inequalities we prove that
$\re Pf(p_{m+1}^+)\ge 2+m-(m+1)\delta$, so again $\|P\|\ge 2+m-(m+1)\delta$.

As $\delta\in(0,1)$ is arbitrary, we conclude that
$$\|P\|\ge 2+m=2+\left\lceil \frac{n-1}2\right\rceil$$ 
which completes the proof of 2(a).

Remark that we have in fact proved also the `only if' parts of the assertions 1 and 2. Indeed, if $Q\ne\emptyset$, then $\io(Q)\ge 0$ and hence any projection $P$ of $C(K)$ onto $q^*C(L)$ satisfies $\|P\|\ge 2+\lceil\frac{0-1}2\rceil=2$. Thus there is no projection of norm $1$.
Further, if $\io(Q)=\infty$, then $\io(Q)\ge n$ for each $n\in\nat$. So any projection $P$ of $C(K)$ onto $q^*C(L)$ satisfies $\|P\|\ge 2+\lceil\frac{n-1}2 \rceil$ for each $n\in\nat$. Therefore there is no bounded linear projection.

We proceed by proving the `if' part of the assertion 1. Suppose that $Q=\emptyset$. We will describe a right inverse for $q$. Define a mapping
$i:L\to K$ by setting
$$i(x)=\begin{cases} x^+, & \mbox{ if }|q^{-1}(x)|=1,\mbox{ i.e., }x^+=x^-,\\
                     x^+, & \mbox{ if }x^-\mbox{ is external in }K,\\
                     x^-, & \mbox{ otherwise}. \end{cases}$$
It is clear that $q\cmp i=\id L$. Moreover, it is easy to check that $i$ is continuous. Indeed, it follows from our assumptions that at least one of the points $x^-$ and $x^+$ is external in $K$ whenever $x^-<x^+$. 

It remains to prove the assertion 2(b). Suppose that $n\in\Nat\cup\{0\}$ and $\io(Q)\le n$. For each $p\in Q$ find an open interval $(\alpha_p,\beta_p)$ in $L$ such that $p$ is one of the endpoints and this interval contains no point $p'\in Q$ with $\io(p',Q)\ge\io(p,Q)$. Such a choice is possible due to the definition of the internal order and due to the fact that $\io(p,Q)<\infty$ for each $p\in Q$. Moreover, it can be easily achieved that $[\alpha_p,\beta_p]$ and $[\alpha_{p'},\beta_{p'}]$ are disjoint whenever $p,p'\in Q$ are two distinct points with $\io(p,Q)=\io(p',Q)$.

For each $p\in Q$ we can choose a continuous function $h_p:L\to[0,1]$ with $h_p(x)=0$ for $x\le \alpha_p$ and $h_p(x)=1$ for $x\ge\beta_p$. 
 
Further, let $Q_e$ be the set of all external points $p\in L$ with $|q^{-1}(p)|>1$. If $p$ is isolated in $[p,\to)$, let $h_p$ be the characteristic function of $(p,\to)$. Otherwise let $h_p$ be the characteristic function of $[p,\to)$.

Now we are going to define a projection of $C(K)$ onto $q^*C(L)$. 
To do it note first that the functions from $C(K)$ which are constant on $[p^-,p^+]$ for each $p\in Q\cup Q_e$ with finitely many exceptions are dense in $C(K)$. So it is sufficient to define the projection on them.
Any function of this kind can be uniquely written in the form
$$f=g+\sum_{j=1}^k c_j g_j,$$
where $g\in q^*C(L)$ and for each $j=1,\dots,k$, $c_j$ is a real number and $g_j\in C(K)$ is such that $g_j(x)=0$ for $x\le p_j^-$ and $g_j(x)=1$ for $x\ge p_j^+$, where $p_1<p_2<\dots<p_k$. If $f$ has this form, set
$$Pf=g+ \sum_{j=1}^k c_j h_{p_j}\cmp q.$$ 
It is easy to see that $P$ defines a linear projection onto $q^*C(L)$. It remains to compute its norm. To this end note that
$$Pf=f+\sum_{j=1}^k c_j (h_{p_j}\cmp q-g_j).$$
Fix any $j\in\{1,\dots,k\}$. If $p_j\in Q$, we have
\begin{alignat*}{3}
h_{p_j}\cmp q(x)-g_j(x)&=0,& x&\in K\setminus(\alpha_{p_j}^+,\beta_{p_j}^-)\cup[p_j^-,p_j^+],\\
|h_{p_j}\cmp q(x)-g_j(x)|&\le 1,& \qquad x&\in[\alpha_{p_j}^+,\beta_{p_j}^-].
\end{alignat*}

If $p_j\in Q_e$, then
$$h_{p_j}\cmp q(x)-g_j(x)=0 \mbox{ for } x\in K\setminus[p_j^-,p_j^+].$$

Further note that $|c_j|=|f(p_j^+)-f(p_j^-)|\le 2\|f\|$.

Now we are ready to estimate the norm of $Pf$. 
Let $x\in K$ be arbitrary. If $x$ does not belong to $[p_j^-,p_j^+]$ for any $j\in\{1,\dots,k\}$, then
\begin{align*}
|Pf(x)|&\le|f(x)|+\sum\{|c_j|: j\in\{1,\dots,k\}, p_j\in Q,  x\in[\alpha_{p_j}^+,\beta_{p_j}^-]\}\\ &\le 
|f(x)|+2(n+1)\|f\|\le(2n+3)\|f\|,\end{align*}
where we used the facts that $\|c_j\|\le 2\|f\|$ and that $x$ can belong to $[\alpha_{p_j}^+,\beta_{p_j}^-]$ only for at most $n+1$ different values for $j$ ($[\alpha_p,\beta_p]$ are pairwise disjoint if the value $\io(p,Q)$ is fixed
and there are only $n+1$ possible values of $\io(p,Q)$).

Next suppose that $x\in[p_j^-,p_j^+]$ for some $j\in\{1,\dots,k\}$. 
There are several possibilities:

(i) $p_j\in Q_e$, $p_j$ is isolated in $[p_j,\to)$. Then $Pf(x)=Pf(p_j^+)$ and
$h_{p_j}\cmp q(p_j^+)-g_j(p_j^+)=0$, therefore
$$|Pf(x)|=|Pf(p_j^+)|\le |f(p_j^+)|+\sum\{|c_i|: i\in\{1,\dots,k\}, p_i\in Q, x\in[\alpha_{p_i}^+,\beta_{p_i}^-]\}\le (2n+3)\|f\|$$

(ii) $p_j\in Q_e$, $p_j$ is not isolated in $[p_j,\to)$. Then $Pf(x)=Pf(p_j^-)$ and $h_{p_j}\cmp q(p_j^-)-g_j(p_j^-)=0$, therefore again
$$|Pf(x)|=|Pf(p_j^+)|\le(2n+3)\|f\|.$$

(iii) $p_j\in Q$, $p_j=\beta_{p_j}$. Then $Pf(x)=Pf(\beta_{p_j}^-)$, hence
$$|Pf(x)|=|Pf(\beta_{p_j}^-)|\le |f(\beta_{p_j}^-)|+\sum\{|c_i|: i\in\{1,\dots,k\}, p_i\in Q, x\in[\alpha_{p_i}^+,\beta_{p_i}^-]\}\le (2n+3)\|f\|$$

(iv) $p_j\in Q$, $p_j=\alpha_{p_j}$. Then similarly
$$|Pf(x)|=|Pf(\alpha_{p_j}^+)|\le |f(\alpha_{p_j}^+)|+\sum\{|c_i|: i\in\{1,\dots,k\}, p_i\in Q, x\in[\alpha_{p_i}^+,\beta_{p_i}^-]\}\le (2n+3)\|f\|$$

Hence we have proved that $\|P\|\le (2n+3)$ which was to be shown.
\end{proof}

\section{Typical metrizable quotients of compact lines}\label{typq}

Let $K$ be a compact space. By Lemma~\ref{typical_is_quotient} we know that typical separable subspace of $C(K)$ is of the form $q^*C(L)$ where $L$ is a metrizable compact space and $q:K\to L$ is a continuous surjection. 

We will extend and modify our terminology in a natural way. By $\Qyuc_\omega(K)$ we will denote, following \cite{AviKal} the set of all metrizable  quotients of $K$. We can view any such quotient in three ways: As a continuous surjection $q:K\to L$, as the respective subset $q^*C(L)$ of $C(K)$ or as the respective equivalence relation on $K$ define by $x\sim y$ if and only if $q(x)=q(y)$. We will identify two surjections generating the same equivalence relation, equivalently defining the same subset of $C(K)$. In other words, $q_1:K\to L_1$ and $q_2:K\to L_2$ will be identified if there is a homeomorphism $h:L_1\to L_2$ such that $q_2=h\cmp q_1$. 

We can define on $\Qyuc_\omega(K)$ a natural partial order (see \cite{Shchepin} or \cite{AviKal}): $q_1\preceq q_2$ if and only if $q_1^*C(\img{q_1}{K})\subs q_2^*C(\img{q_2}{K})$ if and only if the equivalence defined by $q_2$ is contained in the equivalence defined by $q_1$.
This order makes the set $\Qyuc_\omega(K)$ a $\sigma$-complete semilattice (cf \cite[Section 2]{AviKal}).
Remark that if $q_n$, $n\in\Nat$, are elements of $\Qyuc_\omega(K)$, then their supremum can be described either by the closed algebra generated by $\bigcup_n q_n^*C(\img{q_n}{K})$ or by the equivalence relation which is the intersection of  the equivalence relations defined by all the $q_n$.

We will say shortly that {\em typical quotient of $K$ has property $\Pee$} if 
there is a closed cofinal subset $\Qyuc\subs\Qyuc_\omega(K)$ such that each element of $\Qyuc$ has property $\Pee$. It is  equivalent to say that
typical separable subspace of $C(K)$ is of the form $q^*C(L)$ where $q$ has property $\Pee$.

\subsection{Basic facts on order preserving quotients}

Suppose that $K$ is a compact line. Note that $q\in\Qyuc_\omega(K)$ is order preserving if and only if the respective equivalence classes (i.e., the fibers of $q$) are closed intervals. More precisely, if we have an order preserving surjection onto a compact line, then the equivalence classes are clearly closed intervales.
Conversely, if we have an equivalence relation on $K$ whose equivalence classes are closed intervals, then on the quotient we can define a natural compatible order making the quotient map order preserving.

This observation will be used in the proof of the following lemma, which implies the second part of Lemma~\ref{typical_is_quotient}.

\begin{lm}\label{typical_is_increasing} Let $K$ be a compact line. Then typical quotient of $K$ is order preserving.
\end{lm}

\begin{proof} It follows easily from the preceding remarks that the set of all order preserving metrizable quotients is closed in $\Qyuc_\omega(K)$. It remains to prove that it is cofinal.

To this end fix any countable $C\subs C(K)$.  By \cite[Proposition 3.2]{K2006} the linear span of order preserving continuous functions in dense in $C(K)$. So, there is a countable set $C_1$ of order preserving continuous functions such that $C\subs\overline{\Sp C_1}$. Let $S$ be the closed algebra (closed also to taking complex conjugate) generated by $C_1$ and constant functions.

We define an equivalence $\sim_S$ on $K$ such that $x\sim_S y$ if and only if $f(x)=f(y)$ for each $f\in S$. It is clear that $x\sim_S y$ if and only if $f(x)=f(y)$ for each $f\in C_1$. Therefore the equivalence classes are closed intervals. It follows that $S=q^* C(L)$ for an order preserving continuous surjection $q$. This completes the proof.
\end{proof}

The set of all order preserving metrizable quotients of $K$ we will denote by $\Qyuc_\omega^o(K)$. The following lemma is an important tool to generate 
nontrivial order preserving metrizable quotients.

\begin{lm}\label{o-urysohn} Let $K$ be a compact line. Let $a,b\in K$ such that $a\le b$ and the interval $[a,b]$ is $G_\delta$ in $K$. Then there is $q\in\Qyuc_\omega^o(K)$ such that $[a,b]$ is one of the fibers of $q$ (i.e., of the respective equivalence classes).
\end{lm}

\begin{proof} 
We will show that there is an order preserving continuous function $q:K\to \Err$ such that $q^{-1}(0)=[a,b]$. Then $q\in \Qyuc_\omega^o(K)$.

It is enough to prove it in case that either $a=\min K$ or $b=\max K$. Indeed,
if $q_1$ is such a function for $(\leftarrow,b]$ and $q_2$ is such a function for $[a,\to)$, we can take $q=q_1+q_2$.

Suppose that $a=\min K$ (the case $b=\max K$ is analogous). If $b$ is isolated in $[b,\to)$, let $q$ be the characteristic function of $(b,\to)$. Otherwise there is a sequence $b_n\in K$ such that $b_n\searrow b$ (due to the assumption that $[a,b]$ is $G_\delta$). By \cite[Lemma 3.1]{K2006} there is, for each $n\in\Nat$, an order preserving continuous function $f_n:K\to[0,1]$ such that
$f_n=0$ on $(\leftarrow,b]$ and $f_n=1$ on $[b_n,t\to)$. We conclude by setting
$q=\sum_{n=1}^\infty 2^{-n} f_n$.
\end{proof}

The behaviour of (typical) metrizable quotients describes the properties of the given compact space. We illustrate it by two propositions.

\begin{prop} Let $K$ be a compact line. Then $K$ is scattered if and only if $\img qK$ is countable for each $q\in\Qyuc_\omega^o(K)$.
\end{prop}

\begin{proof} If $K$ is scattered, then each continuous image is also scattered.
Moreover, scattered metrizable compact spaces are countable. This proves the `only if' part. To show the `if' part we can use the well-known fact that any non-scattered compact space maps continuously onto $[0,1]$. Fix such a surjection
$q:K\to[0,1]$. By Lemma~\ref{typical_is_increasing} there is a $q'\in\Qyuc_\omega^o(K)$ such that $q\preceq q'$. Clearly $\img{q'}{K}$ is uncountable.
\end{proof}

\begin{prop}\label{aleph0monolitic} Let $K$ be a compact line.
Then $K$ is is  $\aleph_0$-monolithic if and only if 
for all $q\in\Qyuc_\omega^o$ the set $\setof{y\in \img qK}{|q^{-1}(y)|>1}$ is countable.

Moreover, if $K$ is not $\aleph_0$-monolithic, there is $q\in\Qyuc_\omega^o(K)$ such that for each $q'\in\Qyuc_\omega^o(K)$ with $q\preceq q'$ the respective set is uncountable.
\end{prop}

To prove this proposition we will need the following lemmata.

\begin{lm}\label{metrizable} Let $K$ be a metrizable compact line. Then $K$ is order-homeomorphic to a subset of $[0,1]$.
\end{lm}

\begin{proof} Let $U_n$, $n\in\Nat$ form a countable basis of $K$ consisting of open intervals. By \cite[Lemma 3.1]{K2006} there is, for each $n\in\Nat$, an order preserving continuous function $f_n:K\to[0,1]$ such that $f_n(x)=0$ for $x<U_n$ and $f_n(x)=1$ for $x>U_n$. Set $f=\sum_{n=1}^\infty 2^{-n} f_n$. Then $f$ is a continuous strictly increasing function from $K$ to $[0,1]$. This completes the proof.
\end{proof}

\begin{lm}\label{op-extension} Let $K$ be a compact line and $L\subs K$ be a closed subset.
Then for each $q\in \Qyuc_\omega^o(L)$ there is $q'\in\Qyuc_\omega^o(K)$ such that $q'\rest L=q$.
\end{lm}

\begin{proof} By Lemma~\ref{metrizable} we may suppose that $q$ is a continuous order preserving map of $L$ into $[0,1]$. This $q$ can be extended to a continuous order preserving map $q':K\to[0,1]$. The construction of such an extension is given in the proof of Lemma 4.2 of \cite{K2006}.
\end{proof}
\begin{lm}\label{ctble-metr} Let $K$ be a compact space and $q:K\to L$ a continuous surjection of $K$ onto a metrizable compact space $L$. Suppose that there are only countably many $y\in L$ with $|q^{-1}(y)|>1$ and, moreover, each $q^{-1}(y)$ is metrizable. Then $K$ is metrizable as well.
\end{lm}

\begin{proof} We will use the well-known fact that a compact space $K$ is metrizable if and only if there are countably many continuous functions which separate points of $K$. 

So let $g_n$, $n\in\Nat$, be continuous functions on $L$ which separate points of $L$. Further, for each $y\in L$ such that $|q^{-1}(y)|>1$ we can choose continuous functions $h'_{y,n}$, $n\in\Nat$, on $q^{-1}(y)$ which separate points of $q^{-1}(y)$. Further, let $h_{y,n}$ be a continuous extension of $h'_{y,n}$ defined on $K$ (it exists by Tietze's theorem). Then functions $h_{y,n}$ together with $g_n\cmp q$ do separate points of $K$. It follows that $K$ is metrizable.
\end{proof}

\begin{proof}[Proof of Proposition~\ref{aleph0monolitic}.] Suppose that $K$ is not $\aleph_0$-monolithic. Then there is a countable set $C\subs K$ such that $K_1=\overline{C}$ is not metrizable.
As $K_1$ is linearly ordered and separable, we conclude that $K_1$ is first countable. As each point of $C$ is $G_\delta$ in $K_1$, by Lemma~\ref{o-urysohn} we can find, for each $x\in C$, some $q_x\in \Qyuc_\omega^o(K_1)$ such that $q_x^{-1}(q_x(x))=\{x\}$. Let $q_1$ be the supremum of all $q_x$, $x\in C$. Then 
$q_1^{-1}(q_1(x))=\{x\}$ for each $x\in C$.  Now suppose that $y\in \img{q_1}{K_1}$ is such that $|q_1^{-1}(y)|>1$. Then $q_1^{-1}(y)$ is a closed interval, say $[y^-,y^+]$ with $y^-<y^+$.  As $q_1^{-1}(y)\cap C=\emptyset$ and $C$ is dense, we infer that
$(y^-,y^+)=\emptyset$. Hence $q_1^{-1}(y)=\{y^-,y^+\}$ and so it is metrizable.
It follows by Lemma~\ref{ctble-metr} that there must be uncountably many of such points $y$. By Lemma~\ref{op-extension} there is $q\in\Qyuc_\omega^o(K)$ such that $q\rest K_1=q_1$. Clearly $|q^{-1}(y)|>1$ for uncountably many $y\in q(K)$. Moreover, if $q'\in\Qyuc_\omega^o(K)$ is such that $q\preceq q'$, then
$(q')^{-1}(q'(x))=\{x\}$ for each $x\in C$. Therefore by the same reasoning we get that $|(q')^{-1}(y)|>1$ for uncountably many $y\in \img {q'}{K}$.

Now suppose that $K$ is $\aleph_0$-monolithic. Let $q\in\Qyuc_\omega^o(K)$ be arbitrary.
Set $L = \img qK$.
For each $y\in L$ denote by $[y^-,y^+]$ the inverse image $q^{-1}(y)$. Let $K_1=K\setminus\bigcup_{y\in Y}(y^-,y^+)$. Then $\img q{K_1}=L$ and $q\rest K_1$ is at most two-to-one.
Therefore by a claim in the proof of \cite[Proposition 3.4]{K2006} $K_1$ is separable. As $K$ is $\aleph_0$-monolithic, we deduce that $K_1$ is metrizable.
Therefore only for countably many $y\in L$ we can have $y^-<y^+$. This completes the proof.
\end{proof}

We also include the following lemma on lifting closed cofinal subsets. 

\begin{lm}\label{lift-cc} Let $K$ be a compact line and $L\subs K$ be a closed subset.
Suppose that $\See$ is a closed cofinal subset of $\Qyuc_\omega^o(L)$.
Then $\See'=\{q\in \Qyuc_\omega^o(K): q\rest L\in\See\}$ is closed and cofinal
in $\Qyuc_\omega^o(K)$.
\end{lm}

\begin{proof} It is obvious that $\See'$ is closed. We will show it is also cofinal.
Let $q_0\in\Qyuc_\omega^o(K)$ be arbitrary. Using the assumption that $\See$ is cofinal and Lemma~\ref{op-extension} we can construct $q_n,r_n\in\Qyuc_\omega^o(K)$ such that $q_{n-1}\rest L\preceq r_{n}\rest L$, $r_n\rest L\in\See$ and $q_n=\sup\{q_{n-1},r_n\}$. Then $q_0\preceq q_1\preceq q_2\preceq\dots$, so we can take $q=\sup_n q_n$. As
$q_0\rest L\preceq r_1\rest L \preceq q_1\rest L \preceq r_2\rest L\preceq\dots$,
we get that $q\rest L=\sup_n q_n\rest L=\sup_n r_n\rest L\in\See$. Thus $q\in\See'$ which was to be shown.
\end{proof}

\subsection{Typical quotients and internal order}

Let $K$ be a compact line. Denote by $M(K)$ the set of all internal points of uncountable character. Let $q:K\to L$ be a continuous order preserving surjection. We will denote 
$$Q(q)=\{x\in L: x\mbox{ is internal and } |q^{-1}(x)|>1\}.$$
The aim of this section is to prove the following proposition.

\begin{prop}\label{typical-io}
Let $K$ be a compact line. Then for typical quotient $q$ of $K$ 
we have $\io(Q(q))\ge\io(M(K))$.
\end{prop}

This proposition is an immediate consequence of the following lemma.

\begin{lm} Let $K$ be a compact line and $n\in\Nat\cup\{0\}$. If $x\in M(K)$ is such that $\io(x,M(K))\ge n$, then for typical quotient $q\in \Qyuc_\omega^o(K)$ we have $\io(q(x),Q(q))\ge n$.
\end{lm}

\begin{proof} We begin by two observations on order preserving quotients. Let $q:K\to L$ be an order preserving quotient and $L$ be metrizable. Then:
\begin{equation}\label{A}
\mbox{If $x\in K$ has uncountable character, then $|q^{-1}(q(x))|>1$.}
\end{equation}
Indeed, suppose that $x\in K$ has uncountable character in $(\leftarrow,x]$. Then there is a regular uncountable cardinal $\kappa$ and a strictly increasing transfinite sequence $(x_\alpha,\alpha<\kappa)$ in $K$ with limit $x$. As $L$ is metrizable, we get that the transfinite sequence $q(x_\alpha)$ is eventually constant and hence there is some $\alpha<\kappa$ with $q(x_\alpha)=q(x)$.

Further it is easy to check that we have the following:
\begin{equation}\label{B}
\begin{matrix}\mbox{A point $x\in L$ is internal in $L$ if and only if}\\ \mbox{$\min q^{-1}(x)$ is not isolated from the left}\\ \mbox{and $\max q^{-1}(x)$ is not isolated from the right.}\end{matrix}
\end{equation}

Now we proceed to the proof of the lemma. 
The proof will be done by induction.
First we prove it for $n=0$. Suppose that $x\in M(K)$. 
(This means, by definition, that $\io(x,M(K))\ge 0$.)
Let $\See_0(x)$ be the set of all $q\in\Qyuc_\omega^o(K)$ such that
$q^{-1}(q(x))=[a,b]$ such that $a$ is a non-isolated $G_\delta$-point of $(\leftarrow,a]$ and $b$ is a non-isolated $G_\delta$-point of $[b,\to)$.
Then $q(x)\in Q(q)$ for each $q\in \See_0(x)$ (by (\ref{A}) and (\ref{B})).  

So it is enough to check that $\See_0(x)$ is a closed cofinal subset of $\Qyuc_\omega^o(K)$. 

To prove it is closed fix $q_1\preceq q_2\preceq\dots$ in $\See_0(x)$ and denote by $q$ their supremum.
Then $q_k^{-1}(q_k(x))=[a_k,b_k]$ where $a_k$ and $b_k$ have the above properties. Moreover, we have $a_1\le a_2\le\dots$ and $b_1\ge b_2\ge\dots$. Then 
$q^{-1}(q(x))=[a,b]$ where $a=\sup_k a_k$ and $b=\inf_k b_k$. It is clear that $a$ and $b$ have the required property and so $q\in\See_0(x)$.

To prove the cofinality, choose any $q_1\in \Qyuc_\omega^o(K)$. Then $q_1^{-1}(q_1(x))=[u,v]$ for some $u\le x$ and $v\ge x$ such that $u<v$ (by (\ref{A})). As $x$ is internal in $K$, we can choose $a\in [u,x]$ and $b\in [b,v]$ such that $a$ is a non-isolated $G_\delta$-point of $(\leftarrow,a]$ and $b$ is a non-isolated $G_\delta$-point of $[b,\to)$. 
 By Lemma~\ref{o-urysohn} there is $q_2\in \Qyuc_\omega^o(K)$ such that $q_2^{-1}(q_2(x))=[a,b]$. It remains to set $q=\sup\{q_1,q_2\}$.
Then $q\ge q_1$ and $q^{-1}(q(x))=[a,b]$, so $q\in\See_0(x)$. 
This finishes the proof for $n=0$. 

Further, suppose that the statement holds for some $n\ge 0$. Let us prove it  for $n+1$. For each $x\in M(K)$ with $\io(x,M(K))\ge n$ fix a closed cofinal set $\See_n(x)\subs\Qyuc_\omega^o(K)$ such that $\io(q(x),Q(q))\ge n$ for each $q\in\See_n(x)$.  Suppose now that $x\in M(K)$ is such that $\io(x,M(K))\ge n+1$. We let $\See_{n+1}(x)$ be the set of all $q\in \Qyuc_\omega^n(K)$ such that $q^{-1}(q(x))=[a,b]$ where $a,b$ satisfy
\begin{itemize}
	\item $a$ is a non-isolated $G_\delta$-point of $(\leftarrow,a]$ and $b$ is a non-isolated $G_\delta$-point of $[b,\to)$.
  \item The set $\{y<a: \io(y,M(K))\ge n \mbox{ and }q\in \See_n(y)\}$ is nonempty and has supremum $a$.
  \item The set $\{y>b: \io(y,M(K))\ge n \mbox{ and }q\in \See_n(y)\}$ is nonempty and has infimum $b$.
\end{itemize}

It is easy to check that $\See_{n+1}(x)$ is closed. We will show that it is also cofinal by an easy closing-off argument: Let $q_0\in\Qyuc_\omega^o(K)$ be arbitrary. We will construct by induction $a_k,b_k\in K$, $L_k,R_k\subs M(K)$
and $q_k,s_k\in \Qyuc_\omega^o(K)$ for $k\in\Nat$ such that the following conditions are fulfilled:
\begin{itemize}
	\item[(i)] $x\in[a_k,b_k]\subs q_{k-1}^{-1}(q_{k-1}(x))$;
	\item[(ii)] $a_k$ is a non-isolated $G_\delta$-point of $(\leftarrow,a_k]$ and $b_k$ is a non-isolated $G_\delta$-point of $[b_k,\to)$;
	\item[(iii)] $a_k=\sup\{y<a_k, a_k\in M(K), \io(y,M(K))\ge n\}$ 
	\item[(iv)] $b_k=\inf\{y>b_k, b_k\in M(K), \io(y,M(K))\ge n\}$
	\item[(v)] $s_k\ge q_{k-1}$, $s_k^{-1}(s_k(x))=[a_k,b_k]$;
	\item[(vi)] $L_k$ is a countable subset of $M(K)\cap (\leftarrow,a_k)$ such that
	$\sup L_k=a_k$;
	\item[(vii)] $R_k$ is a countable subset of $M(K)\cap (b_k,\to)$ such that
	$\inf R_k=a_k$;
	\item[(viii)] $\io(y,M(K))\ge n$ for each $y\in L_k\cup R_k$;
	\item[(ix)] $q_k\ge s_k$;
	\item[(x)] $q_k\in\bigcap\{\See_n(y): y\in\bigcup_{j=1}^k(L_j\cup R_j)\}$.
\end{itemize}

Let $k\ge 1$ and suppose that $q_{k-1}$ is constructed. Then $q_{k-1}^{-1}(q_{k-1}(x))$ is a closed $G_\delta$ interval containing $x$.
As $\io(x,M(K))\ge n+1$, we can find $a_k$ and $b_k$ such that the conditions (i)--(iv) are fulfilled. Using Lemma~\ref{o-urysohn} we can find $s_k$ satisfying the condition (v) (similarly as in the first induction step).
It follows from conditions (ii)--(iv) that we can find sets $L_k$ and $R_k$ satisfying (vi)--(viii). As $\See_n(y)$ is closed and cofinal for each 
$y\in\bigcup_{j=1}^k(L_j\cup R_j)$ and the latter set is countable, the intersection described in (x) is cofinal and so we can find $q_k$ satisfying (ix) and (x). This completes the construction.

Finaly, set $q=\sup_k q_k=\sup_k s_k$. Then clearly $q\in\See_{n+1}(x)$.
This completes the proof that $\See_{n+1}(x)$ is cofinal.

We have proved that $\See_{n+1}(x)$ is a closed cofinal subset of $\Qyuc_\omega^o(K)$ and for any $q\in\See_{n+1}(x)$ we have $q(x)\in Q(q)$ and $\io(q(x),Q(q))\le n+1$. This completes the proof.
\end{proof}

\subsection{Applications to the continuous SCP}

\begin{tw}\label{thm-negative} Let $K$ be a compact line. If $C(K)$ has the continuous SCP, then $K$ is $\aleph_0$-monolithic and $\io(M(K))<\infty$.
 
If $C(K)$ has moreover the continuous $k$-SCP for some $k<2$, then $M(K)=\emptyset$.
\end{tw}

\begin{proof} The statements concerning the internal order follow immediately
from Proposition~\ref{typical-io} and Lemma~\ref{lines-averaging}.
 
Suppose that $K$ is not $\aleph_0$-monolithic. It follows from Proposition~\ref{aleph0monolitic} that for typical $q\in \Qyuc_\omega^o(K)$ we have $|q^{-1}(y)|>1$ for uncountably many $y\in q(K)$. Fix any such $q$ and
set $L= \img q K$. As $L$ has only countably many external points, there are uncountably many internal points $y\in L$ such that $|q^{-1}(y)|>1$. In other words, the set $Q(q)$ is uncountable. 

Let $y\in Q(q)$ be arbitrary. Suppose that $\io(y,Q(q))=n<\infty$. Then there is a nonempty open interval $(\alpha_y,\beta_y)$ such that $y=\alpha_y$ or $y=\beta_y$ and this interval contains no other points of the same internal order. Moreover, it can be achieved that the intervals $(\alpha_y,\beta_y)$ are mutually disjoint for points of the same internal order. As $L$ is metrizable,
we conclude that there are only countably many $y\in Q(q)$ with $\io(y,Q(q))<\infty$. As $Q(q)$ is uncountable, we get $\io(Q(q))=\infty$, 
therefore $q^*(C(L))$ is not complemented in $C(K)$ by Lemma~\ref{lines-averaging}. This completes the proof.
\end{proof}

We now present the announced example concerning ``unbounded" continuous SCP.

\begin{tw}\label{tunboundedscp}
There exists a compact line $K$ of weight $\aleph_1$, with the following properties.
\begin{enumerate}
	\item[(1)] $C(K)$ has the continuous (even controlled) SCP.
	\item[(2)] For every $\ntr$, $C(K)$ fails to have continuous $n$-SCP.
\end{enumerate}
In other words, whenever $\Ef$ is a continuous chain of separable subspaces of $C(K)$ with $C(K)=\bigcup\Ef$, then we can find $\sett{X_n}{\ntr}\subs \Ef$ such that for each $\ntr$ the space $X_n$ is not $n$-complemented in $C(K)$.
\end{tw}

\begin{proof}
We shall use the duality between linearly ordered sets and $0$-dimensional compact lines. Namely, given a linearly ordered set $X$, let $K(X)$ be the set of all final segments of $X$ ordered by inclusion. Recall that $F\subs X$ is a {\em final segment} if $[x,\rightarrow)\subs F$ whenever $x\in F$. It is clear that $K(X)$ is a $0$-dimensional compact line. Further, every increasing map $\map fXY$ between linearly ordered sets induces a continuous increasing map $\map {K(f)}{K(Y)}{K(X)}$, defined by $K(f)(F) = \inv fF$.

Given an ordinal $\al$, we shall consider the set $\Qyu^\al$ ordered lexicographically. Define
$$S = \setof{x\in \Qyu^{\omega_1}}{|\suppt(x)| < \aleph_0},$$
where $\suppt(x) = \setof{\xi}{x(\xi)\ne 0}$.

Fix a {\em ladder system} $\sett{c_\delta}{\delta\in\lim(\omega_1)}$.
That is, the set $c_\delta$ has order type $\omega$ and $\sup c_\delta = \delta$ for each $\delta \in \lim(\omega_1)$.
Here $\lim(\omega_1)$ denotes, as usual, the set of all positive countable limit ordinals. 

Further, fix a function $\map e{\lim(\omega_1)}\omega$ such that $e^{-1}(n)$ is a stationary set for each $\ntr$. The existence of such a function is a special case of Ulam's Theorem (see e.g. \cite[Ch. II, Thm. 6.11]{Kunen}).

For each $\ntr$ choose $A_n\subs \Qyu$ such that $\io(A_n) = n$.
Define
$$T = \setof{x\in\Qyu^{\omega_1}}{ (\exists\;\delta\in \lim(\omega_1))(\exists\; q\in A_{e(\delta)})\; x = q \cdot \chi_{c_\delta} }.$$
Finally, let $X = S\cup T$ and $K = K(X)$. We claim that $K$ is as required.

For this aim, we first describe a natural continuous chain of complemented separable subspaces whose union is $C(K)$.
Given $\al<\omega_1$, let
$$X_\al = \setof{x\in X}{ \sup(\suppt(x)) < \al }.$$
Notice that $X_\delta = \bigcup_{\xi<\delta}X_\xi$, whenever $\delta$ is a limit ordinal. Clearly, $X = \bigcup_{\al < \omega_1}X_\al$.
Let $K_\al = K(X_\al)$ and let $\map{q_\al}K{K_\al}$ be the increasing quotient induced by inclusion $X_\al \subs X$.
Formally, $q_\al(p) = p \cap X_\al$ for every final segment $p\subs X$.

Observe that $\sett{q_\al}{\al<\omega_1}$ induces a continuous inverse sequence with limit $K$.
In other words, setting $E_\al = {q_\al^*}{C(K_\al)}$, we have that $\sett{E_\al}{\al<\omega_1}$ is a continuous chain of closed separable subspaces of $C(K)$, covering $C(K)$.

Notice that each $q_{\al+1}$ is a retraction, therefore $E_{\al+1}$ is $1$-complemented in $C(K)$.
On the other hand, it is easy to see that, given a limit ordinal $\delta$, the set of internal points of $K_\delta$ which have a non-trvial fiber with respect to $q_\delta$ is naturally order isomorphic to $A_{e(\delta)}$.
Since $\io(A_n)<\infty$ for each $\ntr$, by Lemma~\ref{lines-averaging}, $E_\delta$ is complemented in $C(K)$. This shows that $C(K)$ has the continuous SCP.
Suppose it has continuous $k$-SCP. Then this property must be witnessed by a chain of the form $\sett{E_\al}{\al\in C}$, where $C\subs \omega_1$ is closed and unbounded (it is an immediate consequence of Lemma~\ref{spectral}).
Fix $\ntr$ so that
$$2 + \left\lceil \frac{n-1}2\right\rceil > k.$$
Since $e^{-1}(n)$ is stationary, there exists $\delta\in\lim(\omega_1)$ such that $e(\delta)=n$ and $\delta \in C$.
By Lemma~\ref{lines-averaging}(2)(a), no linear projection onto $E_\delta$ can have norm $\loe k$, a contradiction.
\end{proof}

\section{Scattered compact lines}

In this section we prove a partial converse to Theorem~\ref{thm-negative} within scattered compact lines.

\begin{tw}\label{thm-scattered} Let $K$ be a scattered compact line.
\begin{itemize}
	\item[(1)] The following assertions are equivalent:
	
\begin{itemize}
	\item[(i)] $K$ has a retractional skeleton.
	\item[(ii)] $C(K)$ has a $1$-projectional skeleton.
	\item[(iii)] $C(K)$ has the continuous $k$-SCP for some $k<2$.
	\item[(iv)] $M(K)=\emptyset$.
 
\end{itemize}
 \item[(2)] If $M(K)$ is finite, then $C(K)$ has a $4$-projectional skeleton.
 \item[(3)] If $M(K)$ is countable, then $C(K)$ has the continuous $3$-SCP.
 
\end{itemize}
\end{tw}
 
This theorem is really a partial converse to Theorem~\ref{thm-negative} within scattered spaces. However, it is completely satisfactory. We conjecture that the following is true. Nonetheless, we were not able to prove it.

\begin{conj} Let $K$ be a scattered compact line. Then $C(K)$ has a projectional skeleton if and only if $\io(M(K))<\infty$. Moreover, the constant of such a projectional skeleton can be estimated similarly as in Lemma~\ref{lines-averaging}.
\end{conj}

The rest of this section is devoted to the proof of Theorem~\ref{thm-scattered}.
We start by proving the first assertion. The implications (i)$\Rightarrow$(ii)$\Rightarrow$(iii) are trivial. The implication (iii)$\Rightarrow$(iv) follows from Theorem~\ref{thm-negative}. The implication (iv)$\Rightarrow$(i) will be proved using the following two lemmata.

Before stating the first lemma we need the following definition.
Let us denote by $\Cee$ the smallest class of compact lines containing the singleton such that:
\begin{itemize}
	\item If $K\in\Cee$, then the order inverse $K^{-1}\in\Cee$;
	\item If $\kappa$ is any ordinal, $X_\alpha\in\Cee$ for each $\alpha\le\kappa$ and
\begin{itemize}
	\item $\min X_\alpha$ is isolated in $X_\alpha$ whenever $\alpha\le\kappa$ is of uncountable cofinality, and
	\item $\min X_\alpha$ has countable character in $X_\alpha$ whenever $\alpha\le\kappa$ is a limit ordinal of countable cofinality,
\end{itemize}
 then the lexicographic sum of $(X_\alpha)_{\alpha\le\kappa}$ belongs to $\Cee$.
\end{itemize}

\begin{lm}\label{classC} Let $K$ be a scattered compact line. Then $K$ has no internal points of uncountable character (i.e., $M(K)=\emptyset$) if and only if $K\in\Cee$.
\end{lm}

\begin{proof} The `if' part is obvious. Let us prove the `only if' part.

Suppose that $K$ has no internal points of uncountable character. We will define an equivalence relation $\sim$ on $K$ by setting $x\sim y$ if and only if $[x,y]\in\Cee$. (If $x>y$, then $[x,y]$ means $[y,x]$.)

First remark that if $x<y$ and $x\sim y$, then $x'\sim y'$ for all $x',y'\in[x,y]$. The reason is that the class $\Cee$ is clearly closed under 
taking closed intervals.

Next we will show that $\sim$ is indeed an equivalence relation. Only transitivity requires a proof. Suppose that $x<y<z$ are such that $x\sim y$ and $y\sim z$. There are three possibilities:
\begin{itemize}
	\item $y$ has an immediate predecessor $y^-$. Then $[x,y^-]$ and $[y,z]$ belong to $\Cee$ and $[x,z]$ is their lexicographic sum (indexed by $\{0,1\}$). Thus $[x,z]\in\Cee$, i.e. $x\sim z$.
	\item $y$ has an immediate successor $y^+$. The proof is completely analogous to the previous case.
	\item $y$ is an internal point. Then $y$ has countable character, so there 
	are $x_n\in[x,y)$ such that $x_n\nearrow y$ and $x_0=x$. As $K$ is scattered, we can without loss of generality suppose that $x_n$ is isolated for each $n\in\Nat$. Moreover, $y$ has countable character in $[y,z]$. Then $[x,z]$ is the lexicographic sum (indexed by $[0,\omega]$) of intervals $X_n=[x_n,x_{n+1})$,
$n<\omega$ and $X_\omega=[y,z]$. Thus $[x,z]\in\Cee$, i.e. $x\sim z$.
\end{itemize}
 
Now we will show that all equivalence classes are closed intervals. As we already know that all equivalence classes are convex, it is enough to show that
they contain their endpoints.

To this end fix $a\in K$ and set
$$b=\sup\{x\ge a: a\sim x\}.$$
We want to show that $a\sim b$. If $b$ has an immediate predecessor, then the above sup is obviously max, hence $a\sim b$. So, suppose that $b$ has no immediate predecessor. Then there is a regular cardinal $\kappa$ and points
$a_\alpha$, $\alpha<\kappa$ from $[a,b)$ such that
\begin{itemize}
	\item $a_0=a$;
	\item $a_\alpha<a_\beta$ for $\alpha<\beta<\kappa$;
	\item $a_\gamma=\sup_{\alpha<\gamma} a_\gamma$ for $\alpha<\kappa$ limit;
	\item $a_\alpha$ is an isolated point for each isolated ordinal $\alpha\in[1,\kappa)$;
	\item $b=\sup_{\alpha<\kappa} a_\alpha$.
\end{itemize}
For $\alpha<\kappa$ set $X_\alpha=[a_\alpha,a_{\alpha +1})$ and $X_\kappa=\{b\}$.
Then $X_\alpha\in\Cee$ for each $\alpha\le\kappa$ and $[a,b]$ is the lexicographic sum of $X_\alpha$, $\alpha\le\kappa$. Moreover, if $\alpha<\kappa$ is of uncountable cofinality, then $a_\alpha$ has uncountable character in $(\leftarrow,a_\alpha]$ and hence it is isolated in $X_\alpha$ (as it must be external). Further, if $\alpha<\kappa$ is a limit ordinal of countable cofinality, then $a_\alpha$ is not isolated in $(\leftarrow,a_\alpha]$ and hence it has countable character in $X_\alpha$. It follows that $[a,b]\in\Cee$, so $a\sim b$.

Further, let
$$c=\inf\{x\le a: x\sim a\}.$$
Then again $c\sim a$. Indeed, by the above argument we get that the order inverse $[c,a]^{-1}$ belongs to $\Cee$. Hence $[c,a]\in\Cee$ as well, i.e. $c\sim a$.

Finally let $L=K/\sim$ be the respective quotient. It is clearly a Hausdorff compact scattered line. Moreover, if $a,b\in L$ are such that $a<b$, then 
$(a,b)\ne\emptyset$. Indeed, otherwise denote by $[a^-,a^+]$ the equivalence class corresponding to $a$ and by $[b^-,b^+]$ the equivalence class corresponding to $b$. As $(a^+,b^-)=\emptyset$, we get $a^+\sim b^-$. This is a contradiction. It follows that $L$ is connected. But the only connected scattered space is the singleton. It means that there is only one equivalence class, so $K\in\Cee$. This completes the proof.
\end{proof}
  
\begin{lm}\label{lemma-skeleton} Let $\kappa$ be an ordinal and $X_\alpha$, $\alpha\le\kappa$ be compact lines satisfying 
\begin{itemize}
	\item $X_\alpha$ admits a retractional skeleton for each $\alpha\le\kappa$;
	\item $|X_\alpha|=1$ for each $\alpha\le\kappa$ of uncountable cofinality;
	\item $\min X_\alpha$ has countable character in $X_\alpha$ whenever $\alpha\le\kappa$ is a limit ordinal of countable cofinality.
\end{itemize}
Then their lexicographic sum admits a retractional skeleton as well.
\end{lm}

\begin{proof} For each $\alpha\le\kappa$ let $r_{\alpha,s}$, $s\in S_\alpha$ be a retractional skeleton for $X_\alpha$ (where $S_\alpha$ is a $\sigma$-complete directed index set).

If $\alpha\le\kappa$ is a limit ordinal of countable cofinality, then by our assumption $\min X_\alpha$ has countable character in $X_\alpha$. It follows that there is $s_0\in S_\alpha$ such that $r_{\alpha,s_0}(\min X_\alpha)=\min X_\alpha$. Therefore we may assume without loss of generality that
$r_{\alpha,s}(\min X_\alpha)=\min X_\alpha$ for all $s\in S_\alpha$. (Indeed,
we can replace $S_\alpha$ by $\{s\in S_\alpha: s\ge s_0\}$.) Moreover, we can suppose that this set $S_\alpha$ has a minimum $s_\alpha$ and that $r_{\alpha,s_\alpha}(x)=\min X_\alpha$ for each $x\in X_\alpha$.

Now denote by $K$ the lexicographic sum of $X_\alpha$, $\alpha\le\kappa$ and  we shall define a retractional skeleton for $K$. We start by defining the index set:
\begin{multline*}
\Sigma=\{f: f\mbox{ is a mapping}, \dom(f)\mbox{ is a closed countable subset of $[0,\kappa]$ containing $0$}\\ \mbox{such that each isolated point of $\dom(f)$ is an isolated ordinal}, \\ f(\alpha)\in S_\alpha \mbox{ for }\alpha\in\dom(f) \}.\end{multline*}
The order on $\Sigma$ will be defined as follows:
$$f\le g \Leftrightarrow \dom(f)\subs\dom(g) \mbox{ and }f(\alpha)\le g(\alpha) \mbox{ for each }\alpha\in\dom(f)$$
Then $\Sigma$ is a $\sigma$-complete directed set. Indeed, obviously $\Sigma$ is directed. To prove completeness, fix a sequence $f_1\le f_2\le\dots$ in $\Sigma$. Denote by $C$ the closure in $[0,\kappa]$ of $\bigcup_{n\in\Nat}\dom(f_n)$. Then $C$ has all the properties required for domains of elements of $\Sigma$. We will define the mapping $f$ on $C$ as follows.
If $\alpha\in \bigcup_{n\in\Nat}\dom(f_n)$, we set $f(\alpha)=\sup\{ f_n(\alpha):\alpha\in\dom(f_n)\}$. If $\alpha$ belongs to no $\dom(f_n)$,
we set $f(\alpha)=\min S_\alpha$ (note that necessarily $\alpha$ is a limit ordinal of countable cofinality). It is clear that $f=\sup_n f_n$.

Fix $f\in \Sigma$. We first define a retraction $u_f:[0,\kappa]\to\dom(f)$ by setting $u_f(\alpha)=\max(\dom(f)\cap[0,\alpha])$ for $\alpha\in[0,\kappa]$.

Next we define a retraction $R_f$ on $K$.
Let $x\in K$ be arbitrary. Then there is unique $\alpha\le\kappa$ such that
$x\in X_\alpha$. We set
$$R_f(x)=\begin{cases} r_{\alpha,f(\alpha)}(x) & \mbox{ if }\alpha\in\dom(f) \\
r_{u_f(\alpha),f(u_f(\alpha))}(\max X_{u_f(\alpha)})&\mbox{ if }\alpha\notin\dom(f).\end{cases}$$
It is clear that $R_f\cmp R_f=R_f$ and that $\img{R_f}{K} = \bigcup_{\alpha\in\dom(f)} \img{r_{\alpha,f(\alpha)}}{X_\alpha}$. 

We will check that $R_f$ is continuous. First note that it is continuous when restricted to any $X_\alpha$. Indeed, if $\alpha\in\dom(f)$, then it follows from the continuity of $r_{\alpha,f(\alpha)}$, and if $\alpha\notin\dom(f)$, then $R_f$ is constant on $X_\alpha$. So, $R_f$ is continuous (with respect to $K$) at each point of $X_\alpha\setminus\{\min X_\alpha\}$ for each $\alpha\le\kappa$. It remains to prove the continuity at $\min X_\alpha$ for each $\alpha$.

If $\alpha$ is an isolated ordinal, then $X_\alpha$ is clopen in $K$,
so $R_f$ is continuous at $\min X_\alpha$. If $\alpha$ is a limit ordinal such that $\alpha\notin\dom(f)$, then $R_f$ is constant on a neighborhood of $\min X_\alpha$ (as the function $u_f$ is constant on a neighborhood of $\alpha$).

Finally suppose that $\alpha$ is a limit ordinal such that $\alpha\in\dom(f)$.
Then, by the definition of $\Sigma$, $\alpha$ is an accumulation point of $\dom(f)$ and so, as $\dom(f)$ is countable, $\alpha$ has necessarily countable cofinality. Therefore there are $\alpha_n\in\dom(f)$ for $n\in\Nat$ such that $\alpha_n\nearrow \alpha$. Then $\min X_{\alpha_n}\nearrow \min X_\alpha$.
Moreover, if $x\in[\min X_{\alpha_n},\min X_\alpha)$, then $R_f(x)\in [\min X_{\alpha_n},\min X_\alpha)$ as well, and hence the limit of $R_f$ at $\min X_\alpha$ from the left is $\min X_\alpha$. Further, $R_f(\min X_\alpha)=\min X_\alpha$ (by the second paragraph of the proof), so $R_f$ is continuous at $\min X_\alpha$ from the left. The continuity from the right follows from the continuity with respect to $X_\alpha$.

This completes the proof that $R_f$ is a continuous retraction. Moreover,
its range is metrizable (by the above it is a countable union of metrizable sets
and so it has countable network).

Next we will show that the retractions $R_f$ are compatible, i.e. $R_f\cmp R_g=R_g\cmp R_f=R_f$ whenever $f\le g$. So, choose any $f\le g$ and $x\in K$.
Fix $\alpha\le\kappa$ such that $x\in X_\alpha$. There are four possibilities:
\begin{itemize}

\item $\alpha\in\dom(f)$: Then $\alpha\in\dom(g)$ as well and $f(\alpha)\le g(\alpha)$,
so $R_f(R_g(x))=R_g(R_f(x))=R_f(x)$ by the compatibility of the retractions
$r_{\alpha,s}$, $s\in S_\alpha$.

\item $\alpha\in\dom(g)\setminus\dom(f)$: Let $\beta=u_f(\alpha)$. Then $\beta<\alpha$ 
and $\beta\in\dom(f)$. Then $R_f(x)=r_{\beta,f(\beta)}(\max X_\beta)$. Further, $R_g(x)\in X_\alpha$ and hence $R_f(R_g(x))=R_f(x)$. Moreover,
$R_g(R_f(x))=R_g(r_{\beta,f(\beta)}(\max X_\beta))=
r_{\beta,g(\beta)}(r_{\beta,f(\beta)}(\max X_\beta))=r_{\beta,f(\beta)}(\max X_\beta)=R_f(x)$.

\item $\alpha\notin\dom(g)$ and $\beta=u_g(\alpha)\in \dom(f)$: 
Then $R_f(x)=r_{\beta,f(\beta)}(\max X_\beta)$ and $R_g(x)=r_{\beta,g(\beta)}(\max X_\beta)$. As $f(\beta)\le g(\beta)$, we conclude
that $R_f(R_g(x))=R_g(R_f(x))=R_f(x)$.

\item $\alpha\notin\dom(g)$ and $\beta=u_g(\alpha)\notin \dom(f)$: 
Then $\gamma=u_f(\alpha)=u_f(\beta)<\beta$. Further,
$R_f(x)=r_{\gamma,f(\gamma)}(\max X_\gamma)$. As $R_g(x)\in X_\beta$, we get
$R_f(R_g(x))=R_f(x)$. Moreover, $R_g(R_f(x))=R_f(x)$ as $f(\gamma)\le g(\gamma)$.
\end{itemize}

We proceed by proving the continuity of the constructed family of retractions.
Fix $f_1\le f_2\le\cdots$ in $\Sigma$ and set $f=\sup_n f_n$. Let $x\in X$ be
arbitrary. We will show that $R_{f_n}(x)\to R_f(x)$.
Let $\alpha\in[0,\kappa]$ be such that $x\in X_\alpha$. There are three possibilities:
\begin{itemize}

\item $\alpha\in\dom(f_k)$ for some $k\in\Nat$: Then $\alpha\in\dom(f_n)$ for $n\ge k$ and $\alpha\in\dom(f)$ as well. Further, $f_n(\alpha)\nearrow f(\alpha)$. 
So, $R_{f_n}(x)=r_{\alpha,f_n(\alpha)}(x)\to r_{\alpha,f(\alpha)}(x)=R_f(x)$.

\item $\alpha\in\dom(f)$ but $\alpha\notin\dom(f_n)$ for any $n$: Then $\alpha$ is a limit ordinal of countable cofinality and $f(\alpha)=\min S_\alpha$. Therefore
$R_f(x)=\min X_\alpha$. Further, there are $\alpha_n\in\dom(f_n)$ such that $\alpha_n\nearrow\alpha$. Then $R_{f_n}(x)\in[\min X_{\alpha_n},\min X_\alpha)$,
therefore $R_{f_n}(x)\to \min X_\alpha=R_f(x)$.

\item $\alpha\notin\dom(f)$: Set $\beta=u_f(\alpha)$. Then $\beta\in\dom(f)$ and $R_f(x)\in X_\beta$.
By the already proved cases, it folllows that $R_{f_n}(R_f(x))\to R_f(R_f(x))$.
However, $R_{f_n}(R_f(x))=R_{f_n}(x)$ by the compatibility condition and $R_f(R_f(x))=R_f(x)$. So, $R_{f_n}(x)\to R_f(x)$.
\end{itemize}

Finally, it remains to prove that $R_{f}(x)\to x$. So, fix any $x\in K$ and let $\alpha\in[0,\kappa]$ be such that $x\in X_\alpha$. There are two possibilities:
\begin{itemize}

\item $\alpha$ is either isolated or of countable cofinality: Then there is $f_0\in \Sigma$ such that $\alpha\in\dom(f_0)$. If $f\ge f_0$, then $\alpha\in\dom(f)$ and $R_f(x)=r_{\alpha,f(\alpha)}(x)$, so $R_f(x)\to x$ as $r_{\alpha,s}$, $s\in S_\alpha$ is a retractional skeleton of $X_\alpha$.

\item $\alpha$ is of uncountable cofinality: Then $X_\alpha=\{x\}$ and $(\leftarrow,x]$ is a clopen interval. Fix any $y<x$. Then there is an isolated ordinal $\beta<\alpha$ such that $\min X_\beta>y$. If $f\in \Sigma$ is such that
$\beta\in\dom(f)$, then $R_f(x)\in[\min X_\beta,x)$. This proves that $R_f(x)\to x$ and completes the proof.
\end{itemize}
\end{proof} 

Now we are able to prove the implication (iv)$\Rightarrow$(i) from the assertion (1) of Theorem~\ref{thm-scattered}. Indeed, let $K$ be a scattered compact line such that $M(K)=\emptyset$. By Lemma~\ref{classC} the space $K$ belongs to the class $\Cee$. By Lemma~\ref{lemma-skeleton} it follow that any element of $\Cee$ admits a retractional skeleton. 

\begin{proof}[Proof of assertion (2) of Theorem~\ref{thm-scattered}]
Assume that $K$ is a scattered compact line such that $M(K)$ is finite.
If $M(K)=\emptyset$, we can conclude by the already proved assertion (1). So suppose that $M(K)\ne\emptyset$. Let $M(K)=\{a_1,\dots,a_n\}$, where $n\in\Nat$
and $a_1<a_2<\dots<a_n$. Let $L$ be the compact line made from $K$ by duplicating each point of $M(K)$ (i.e., each of the points $a_i$ is replaced by a pair $a_i^-<a_i^+$. Then $L$ is scattered and $M(L)=\emptyset$, hence $C(L)$ has a $1$-projectional skeleton by the assertion (1). We will describe an isomorphism of $C(K)$ and $C(L)$.

For each $i\in\{1,\dots,n\}$ the point $a_i$ is internal in $K$, 
hence we can find isolated points $b_i,c_i\in K$ for $i=1,\dots,n$ such that
$$b_1<a_1<c_1<b_2<a_2<c_2<\dots<b_n<a_n<c_n.$$
Moreover, as each $a_i$ has uncountable character, the points $b_i$ and $c_i$ can be moreover chosen in such a way that at least one of the intervals
$$(\leftarrow,b_1),(c_1,b_2),(c_2,b_3),\dots,(c_n,\rightarrow)$$
is infinite. Let us fix such an interval and denote it by $I$.
We can find in $I$ a one-to-one sequence $(y_k)_{k=1}^\infty$ of isolated points which converges to a point $y\in I$ (note that $I$ is closed as the points $b_i$ and $c_i$ are isolated). 

Now we are able to define an isomorphism $T:C(L)\to C(K)$ by the following formula

$$T(f)(x)=
\begin{cases} 
\frac12(f(a_i^-)+f(a_i^+)),& x=a_i,\; i=1,\dots,n,\\
f(x)+\frac12(f(a_i^+)-f(a_i^-)),& x\in [b_i,a_i),\; i=1,\dots,n,\\
f(x)+\frac12(f(a_i^-)-f(a_i^+)),& x\in (a_i,c_i],\; i=1,\dots,n,\\
f(a_{(i+1)/2}^-), & x=y_{2i-1},\; i=1,\dots,n,\\
f(a_{i/2}^+), & x=y_{2i},\; i=1,\dots,n, \\
f(y_{k-2n}), & x=y_k,\; k>2n, \\
f(x), & \mbox{otherwise}.
\end{cases}$$

It is easy to check that $T$ is a linear bijection of $C(L)$ and $C(K)$. Moreover, clearly $\|T\|\le 2$ and $\|T^{-1}\|\le 2$. Hence $C(K)$ has a $4$-projectional skeleton.
\end{proof}

\begin{proof}[Proof of assertion (3) of Theorem~\ref{thm-scattered}]
Suppose that $K$ is a scattered compact line and $M(K)$ is nonempty and countable. (The case when $M(K)$ is empty is covered by assertion (1).)
As $M(K)$ is countable and each point of $M(K)$ has uncountable character,
necessarily $\io(M(K))=0$. We will show that for typical quotient of $K$ we have
$\io(Q(q))=0$. Then we will be able to conclude by Lemma~\ref{lines-averaging}.

We set
$$\cal I=\{[a,b]: [a,b]\mbox{ is  maximal such that }(a,b)\cap M(K)=\emptyset\}.$$
If $[a,b]$ and $[c,d]$ are distinct elements of $\cal I$, obviously either $b\le c$ or $d\le a$. Further, if $[a,b]\in \cal I$, then $a,b\in\{\min K,\max K\}\cup\overline{M(K)}$. As $K$ is scattered and $M(K)$ countable, the closure
$\overline{M(K)}$ is countable as well. It follows that $\cal I$ is countable.
Finally, $\cal I$ covers $K$. Indeed, let $x\in K$ be arbitrary. Then $x\in[a,b]\in \cal I$, where
\begin{align*}
a&=\max  \{y\in\{\min K,\max K\}\cup\overline{M(K)}:y\le x\},\\
b&=\min  \{y\in\{\min K,\max K\}\cup\overline{M(K)}:y\ge x\}.
\end{align*}
(Note that it may happen that $a=b$.)

As the set $\{\min K,\max K\}\cup\overline{M(K)}$ is countable, there is $q_0\in\Qyuc_\omega^o(K)$ which is one-to-one on this set. Further, fix $[a,b]\in \cal I$. As $(a,b)\cap M(K)=\emptyset$, we have $M([a,b])=\emptyset$,
so $[a,b]$ admits a retractional skeleton by assertion (1). In particular,
typical quotient $q$ of $[a,b]$ satisfies $Q(q)=\emptyset$ by Lemma~\ref{lines-averaging}. Fix $\See_{[a,b]}$ a closed cofinal subset of $\Qyuc_\omega^o([a,b])$ such that for each $q\in\See_{[a,b]}$ we have $Q(q)=\emptyset$. Let
$$\See=\{q\in\Qyuc_\omega^o(K): q\succeq q_0\ \&\ q\rest [a,b]\in \See_{[a,b]}\mbox{ for each }[a,b]\in\cal I\}$$
By Lemma~\ref{lift-cc} it is a closed cofinal subset of $\Qyuc_\omega^o(K)$.
Let $q\in\See$ be arbitrary. Then $Q(q)\subs \img q{M(K)}$. We will show that $\io(\img q{M(K)})=0$. Let $x\in M(K)$ be arbitrary. As $x$ is internal point of $K$ and $\io(x,M(K))=0$, either there is $y<x$ such that $(y,x)\cap M(K)=\emptyset$
or there is $y>x$ such that $(x,y)\cap M(K)=\emptyset$. Fix such a $y$. Let $[a,b]$ be an element of $\cal I$ containing $[x,y]$. Then $a<b$ and $x$ is one of the endpoints. Further, as $q\succeq q_0$, we have $q(a)<q(b)$. Finally, $(q(a),q(b))\cap \img q{M(K)}=\emptyset$. It follows that $\io(q(x),\img q{M(K))}=0$.
This completes the proof.
\end{proof}

\section{Examples: subspaces of Plichko spaces}

We say that a Banach space is \emph{Plichko} if it admits a commutative projectional skeleton. Further, it is called \emph{$1$-Plichko} if it admits a commutative $1$-projectional skeleton. Here the word \emph{commutative} means that $P_S P_T=P_T P_S$ for any $S,T\in\Ef$ (we use the notation from the definition of a projectional skeleton given in the introductory section).

This is not the original definition of Plichko and $1$-Plichko spaces (see \cite{Kalenda2000}) but it is equivalent to the original one (see \cite[Theorem 27]{skeletons}). It follows that, if $E$ has density $\aleph_1$, then it is Plichko ($1$-Plichko) whenever it has a projectional skeleton ($1$-projectional skeleton). Indeed, in this case the projectional skeleton can be indexed by a well-ordered set $[0,\omega_1)$, so it may be commutative.

It was an open problem whether a subspace of a Plichko space is again Plichko. This question was answered in the negative by the second author in \cite[Theorem 5.1]{K2006} where he constructed a counterexample. In this section we show that, within spaces of continuous functions on compact scattered lines, the property of being a non-Plichko subspace of a Plichko space is quite frequent.

\begin{tw} Let $K$ be a compact scattered line of cardinality $\aleph_1$. Then the following holds:
\begin{itemize}
	\item $C(K)$ is isometric to a subspace of a $1$-Plichko space.
	\item If $\io(M(K))=\infty$, then $C(K)$ is not Plichko.
\end{itemize}
\end{tw}

\begin{proof} The second statement follows immediately from Theorem~\ref{thm-negative}. Let us show the first statement. Let $K$ be a compact scattered line of cardinality $\aleph_1$. Let $L$ be the compact line made from $K$ by duplicating all the points of $M(K)$ (cf. the proof of the assertion (2) of Theorem~\ref{thm-scattered} above) and let $\varphi:L\to K$ be the canonical order preserving surjection which glues back the duplicated points.
Then $M(L)=\emptyset$, hence $L$ has a retractional skeleton by Theorem~\ref{thm-scattered}. As $L$ has cardinality $\aleph_1$, $C(L)$ is $1$-Plichko. Moreover, as $K$ is a continuous image of $L$, $C(K)$ is isometric to a subspace of $C(L)$, namely to
$\{f\circ \varphi: f\in C(K)\}.$
\end{proof}

\begin{tw} For each $n\in \Nat\cup\{-1,0,\infty\}$ there is a compact scattered line $K$ of cardinality $\aleph_1$, such that $\io(M(K))=n$. Moreover, there is a compact scattered line $K$ of cardinality $\aleph_1$ and $x\in M(K)$ such that $\io(x,M(K))=\infty$.
\end{tw}

\begin{proof} If $n=-1$ we can set $K=[0,\omega_1]$. If $n=0$, we can set for example $K=\omega_1+1+\omega^{-1}$ or $K=\omega_1+1+\omega_1^{-1}$. 

We proceed by induction. Let $n\in \Nat\cup\{0\}$  and let $L$ be a compact scattered line of cardinality $\aleph_1$ with $\io(M(L))=n$. Let $K$ be the compact line made either from $\omega_1+1+\omega^{-1}$ or from  $\omega_1+1+\omega_1^{-1}$ by replacing each isolated point with a copy of $L$.
It is clear that then $\io(M(K))=n+1$.

Further, let $K_n$, $n\in\Nat$ be such that $\io(K_n)=n$. Let $\tilde K$ be the lexicographic sum of $K_n$, $n\in\Nat$, along $[1,\omega)$. Let $K$ be the one-point compactification of $\tilde K$ (made by adding the endpoint). Then $\io(M(K))=\infty$ and $\io(x,M(K))<\infty$ for each $x\in M(K)$.

Finally, let $A_n$ be the set of all countable ordinals of the form $\lambda+n$, where $\lambda$ is a limit ordinal. Let $K$ be the space made from
$\omega_1+1+\omega_1^{-1}$ by replacing,  for each $n\in\Nat$, every element of $A_n\cup A_{n}^{-1}$ with $K_n$. Then $\io(x,M(K))=\infty$ where $x$ is the ``middle point'' of $K$. A similar example can be constructed starting from $\omega_1+1+\omega^{-1}$.  
\end{proof}

\end{document}